\documentclass[10pt,leqno]{amsart}
\usepackage{graphicx}
\usepackage{amsmath, amssymb, amsfonts, amsthm, stmaryrd, url, hyperref, indentfirst, bm, bbm, extarrows, tikz-cd, etoolbox}
\usepackage[all]{xy}

\newtheorem{theorem}{Theorem}[section]

\theoremstyle{definition}
\newtheorem{defn}[theorem]{Definition}

\newtheorem{rmk}[theorem]{Remark}

\numberwithin{equation}{theorem}

\newcommand{\AAA}{\mathbb{A}}
\newcommand{\CC}{\mathbb{C}}

\newcommand{\FF}{\mathbb{F}}

\newcommand{\QQ}{\mathbb{Q}}

\newcommand{\ZZ}{\mathbb{Z}}

\newcommand{\calL}{\mathcal{L}}

\newcommand{\calO}{\mathcal{O}}

\DeclareMathOperator{\ad}{ad}

\DeclareMathOperator{\Fil}{Fil}

\DeclareMathOperator{\Frob}{Frob}
\DeclareMathOperator{\Fss}{F-ss}
\DeclareMathOperator{\Gal}{Gal}
\DeclareMathOperator{\GL}{GL}

\DeclareMathOperator{\gr}{gr}

\DeclareMathOperator{\HT}{HT}

\DeclareMathOperator{\M}{M}

\DeclareMathOperator{\rec}{rec}

\DeclareMathOperator{\Res}{Res}

\DeclareMathOperator{\Ss}{ss}

\DeclareMathOperator{\WD}{WD}

\title{A Rank-Two Case of Local-Global Compatibility for $l=p$}
\author{Yuji Yang}
\address{Beijing International Center for Mathematical Research, Peking University}
\email{yujiy@pku.edu.cn}
\date{}

\begin{document}
\begin{abstract}
We prove the classical $l=p$ local-global compatibility conjecture for certain regular algebraic cuspidal automorphic representations of weight $0$ for $\GL_2$ over CM fields. Using an automorphy lifting theorem, we show that if the automorphic side comes from a twist of Steinberg at $v\mid l$, then the Galois side has nontrivial monodromy at $v$. Based on this observation, we will give a definition of the Fontaine--Mazur $\calL$-invariants attached to certain automorphic representations.
\end{abstract}
\maketitle
\tableofcontents
\section{Introduction}
The goal of this work is to prove a new case of local-global compatibility for automorphic Galois representations over CM fields at $l=p$. Let $F$ be a CM field and let $G_F=\Gal(\overline{F}/F)$ denote the absolute Galois group. Let $\pi$ be a regular algebraic cuspidal automorphic representation of $\GL_n(\AAA_F)$. Fix a prime $l$ and an isomorphism $\iota:\overline{\QQ}_l\cong\CC$. Conjecturally there exists a continuous semisimple Galois representation $r_{\iota}(\pi):G_F\to\GL_n(\overline{\QQ}_l)$ that is potentially semistable (which is known to be equivalent to de Rham) above $l$, such that we have an isomorphism of Weil--Deligne representations
$$\WD(r_{\iota}(\pi)|_{G_{F_v}})^{\Fss}\otimes_\iota\CC\cong\rec_{F_v}(\pi_v\otimes|\det|^{\frac{1-n}{2}})$$
for any finite place $v$ in $F$. Here on the automorphic side $\pi_v$ is the local component of $\pi=\otimes'_v\pi_v$ at $v$, and $\rec_{F_v}$ is the local Langlands correspondence for $F_v$ normalized as in \cite{HT01}; on the Galois side $G_{F_v}=\Gal(\overline{F}_v/F_v)$ is the decomposition group at $v$, and $\WD$ is a functor taking a continuous semisimple Galois representation to a Weil--Deligne representation by Grothendieck's $l$-adic monodromy theorem when $v\nmid l$ and by a construction of Fontaine \cite{Fon94} when $v\mid l$ and $r_{\iota}(\pi)|_{G_{F_v}}$ is potentially semistable. We may then study the above isomorphism for all finite places not dividing $l$ and for all places above $l$ separately. The former is called the $l\ne p$ case and the latter is called the $l=p$ case. 

From now on we will focus on the $l=p$ case. If $\pi$ is polarizable, the $l=p$ local-global compatibility has been established by work of many people \cite{BLGGT12, BLGGT14a, Car14}, etc. For non-polarizable $\pi$, the representations $r_{\iota}(\pi)$ have been constructed by \cite{HLTT16} and \cite{Sch15} using different methods. Recently, A'Campo proved that under certain conditions the representation $r_{\iota}(\pi)$ is de Rham with specified Hodge--Tate weights \cite[Theorem~4.3.3]{ACa23}, and Hevesi proved that under the same conditions the $l=p$ local-global compatibility holds up to semisimplification \cite[Theorem~1.1]{Hev23}. This means that to prove the $l=p$ local-global compatibility, we only need to compare the monodromy operators of the two Weil--Deligne representations. It turns out when $\pi$ is of rank $2$ and weight $0$, we can solve the problem most of the time.

\begin{theorem}\label{density}
Let $F$ be a CM field and let $\pi$ be a regular algebraic cuspidal automorphic representation of $\GL_2(\AAA_F)$ of weight $0$. There is a set of rational primes $l$ of Dirichlet density one such that for any isomorphism $\iota:\overline{\QQ}_l\cong\CC$ and any finite place $v\mid l$ in $F$, we have
$$\WD(r_{\iota}(\pi)|_{G_{F_v}})^{\Fss}\otimes_\iota\CC\cong\rec_{F_v}(\pi_v\otimes|\det|^{-1/2}).$$
\end{theorem}

\begin{rmk}
For simplicity, here we present a compromised version of our main theorem. For the full version with more technical conditions, see Theorem \ref{main}.
\end{rmk}

\hspace*{\fill} \\
\textbf{Method of proof.}
By \cite[Theorem~1.1]{Hev23}, it suffices to show that when $\pi_v$ is special (i.e. a twist of Steinberg), the Galois representation $r_{\iota}(\pi)|_{G_{F_v}}$ has nontrivial monodromy. Assume otherwise. Combined with \cite[Theorem~4.3.3]{ACa23}, after a base change we may assume that $r_{\iota}(\pi)|_{G_{F_v}}$ is crystalline. Next we apply a potential automorphy theorem to find an automorphic representation $\pi_1$ that is unramified above $l$ such that there is a congruence between mod $l$ Galois representations $\overline{r}_{\iota}(\pi)\cong\overline{r}_{\iota}(\pi_1)$ when restricted to the absolute Galois group of some finite extension of $F$. Then we are able to apply an automorphy lifting theorem to find an automorphic representation $\pi_2$ that is unramified above $l$ such that $r_{\iota}(\pi)\cong r_{\iota}(\pi_2)$ when restricted to the absolute Galois group of some field extension of $F$. From there we deduce a contradiction with genericity.

The idea to approach the local-global compatibility conjecture using automorphy lifting theorems is inspired by Luu \cite{Luu15}. This method has been used in \cite{AN19} and \cite{Yan21} to prove cases of the $l\ne p$ local-global compatibility, where the isomorphism is known up to semisimplification by the main result of \cite{Var14}. However, in the $l=p$ case there is extra subtlety coming from the construction of the $\WD$ functor. We need to use the fact that this construction is independent of field extensions in some sense. See the next chapter for a review with more details.



Our method is confined to rank two and weight $0$. The proof of our potential automorphy theorem \cite[Theorem~3.9]{AN19} uses a moduli space of Hilbert--Blumenthal abelian varieties which realizes a fixed mod $l$ Galois representation, as in Taylor's original work on the subject \cite{Tay02}. This method only works for $\GL_2$. On the other hand, our automorphy lifting theorem \cite[Theorem~6.1.1]{ACC+18} (of Fontaine--Laffaille version) does not allow a change of weight, and our potential automorphy theorem produces an automorphic representation of weight $0$ as an input of the automorphy lifting theorem. Thus we are also confined to weight $0$. 

When dealing with the $l\ne p$ local-global compatibility, \cite{AN19} also has the restriction to rank two and weight $0$. By using the ordinary automorphy lifting theorem \cite[Theorem~6.1.2]{ACC+18}, \cite{Yan21} removes the weight $0$ condition at the cost of imposing an ordinarity assumption on the automorphic representation at the prime $l$ (the residue characteristic of the coefficients). However, this can not be translated to the $l=p$ local-global compatibility problem, since the ordinary automorphy lifting theorem does not guarantee that the lift $\pi_2$ is unramified above $l$, which is essential for our proof.

To go beyond rank two, we need to upgrade our potential automorphy theorems. When $\pi$ is polarizable such theorems are known \cite{HSBT10, BLGHT11, BLGGT14b} by using Dwork family of motives. In the non-polarizable case, Qian proved potential automorphy theorems under general settings \cite[Theorem~1.1]{Qia23} using Dwork families. Recently in \cite{Mat23}, based on \cite{BLGGT14b, Qia23}, Matsumoto proved many versions of potential automorphy theorems, and through a careful investigation of the Jordan canonical form of the Weil--Deligne representations, he generalized the $l\ne p$ local-global compatibility results \cite{AN19, Yan21} to $\GL_n$. If this method can be translated to the $l=p$ case, then we believe analogous versions of our main theorems can be proved in higher rank.

\hspace*{\fill} \\
\textbf{Notations.} 
Throughout $F$ will be a CM field. If $v$ is a place of $F$, $F_v$ is the completion of $F$ at $v$. Let $\AAA_F$ be the ring of adeles of $F$. Fix $\overline{F}$ an algebraic closure of $F$. We write $G_F=\Gal(\overline{F}/F)$ for the absolute Galois group, and similarly $G_{F_v}=\Gal(\overline{F}_v/F_v)$.

For a non-archimedean local field $K$, let $I_K$ be the inertia subgroup and let $W_K$ be the Weil group.

We write $\epsilon_l$ for the $l$-adic cyclotomic character. We normalize our Hodge--Tate weights such that $\epsilon_l$ has Hodge–Tate weight $-1$. Let $B_{\text{dR}}$, $B_{\text{st}}$, $B_{\text{cris}}$ be Fontaine's rings of de Rham, semistable, and crystalline periods, respectively. We take $K$ and $E$ to be finite extensions of $\QQ_l$, with $E$ sufficiently large to contain the images of all embeddings $\tau:K\hookrightarrow\overline{\QQ}_l$. Let $K_0$ be the maximal unramified extension of $\QQ_l$ inside $K$. If $\rho:G_K\to\GL(V)\cong\GL_d(E)$ is a de Rham representation, we write $D_{\text{dR}}(\rho)=D_{\text{dR}}(V)=(B_{\text{dR}}\otimes_{\QQ_l}V)^{G_K}$, which is a $K\otimes_{\QQ_l}E$-module. We define $D_{\text{st}}$ and $D_{\text{cris}}$ analogously but they are both $K_0\otimes_{\QQ_l}E$-modules. We write $\HT_\tau(\rho)$ for the multiset of $\tau$-labelled Hodge--Tate weights of $\rho$, which are the degrees $i$ at which $\gr^i(D_{\text{dR}}(V)\otimes_{K_0\otimes E,\tau\otimes 1}E)\ne0$, counted with multiplicities equal to the $E$-dimensions of the corresponding graded pieces. 

Let $\pi$ be a regular algebraic cuspidal automorphic representation of $\GL_2(\AAA_F)$. We say that $\pi$ has weight $0$ if it has the same infinitesimal character as the trivial (algebraic) representation of $\Res_{F/\QQ} \GL_2$. For any finite place $v$ in $F$, $\pi_v$ is the local component of $\pi=\otimes'_v\pi_v$ at $v$, and $\rec_{F_v}$ is the local Langlands correspondence for $F_v$ of \cite{HT01}.

For a Weil--Deligne representation $(r,N)$, its Frobenius semisimplification is $(r,N)^{\Fss}=(r^{\Ss},N)$, and its semisimplification is $(r,N)^{\Ss}=(r^{\Ss},0)$. 

\hspace*{\fill} \\
\textbf{Acknowledgements.} 
I would like to thank Patrick Allen for bringing my attention to this project and answering my questions in the early stage with great patience. I would like to thank Liang Xiao and Yiwen Ding for their helpful suggestions on an early draft of this paper. I would like to thank Shanxiao Huang and Yiqin He for helpful discussions on $p$-adic Hodge theory and the $\calL$-invariants. 

\section{Some Important Results}

We first recall Fontaine's construction of the $\WD$ functor following \cite{Fon94, All16}. Let $K$ and $E$ be finite extensions of $\QQ_l$ with $E$ sufficiently large. 

For fixed finite extension $L/K$, we write $G_{L/K}=\Gal(L/K)$ and make the following definition.
\begin{defn}
A \emph{$(\varphi,N,G_{L/K})$-module} is a finite free $L_0\otimes_{\QQ_l}E$-module together with operators $\varphi$ and $N$, and a $\Gal(L/K)$-action, satisfying the following:
\begin{itemize}
    \item $\varphi$ is $E$-linear and $L_0$-semilinear (this means $\varphi(ax)=\sigma(a)\varphi(x)$ for any $a\in L_0$ and $x\in D$, where $\sigma\in\Gal(L_0/\QQ_l)$ is the absolute arithmetic Frobenius);
    \item $N$ is $L_0\otimes_{\QQ_l} E$-linear;
    \item $N\varphi=l\varphi N$;
    \item the $\Gal(L/K)$-action is $E$-linear and $L_0$-semilinear and commutes with $\varphi$ and $N$.
\end{itemize}
\end{defn}
We first associate a Weil--Deligne representation $(r_D, N_D)$ to a $(\varphi,N,G_{L/K})$-module $D$. Let $N_D=N$. Extend the action of $\Gal(L/K)$ to $W_K$ by letting $I_L$ act trivially. For $w\in W_K$, let $\alpha(w)\in\ZZ$ be such that the image of $w$ in $W_K/I_K$ is $\sigma^{-\alpha(w)}$, where $\sigma$ is the image of the absolute arithmetic Frobenius. We then define an $L_0\otimes_{\QQ_l}E$-linear action $r_D$ of $W_K$ on $D$ by $r_D(w)=w\varphi^{\alpha(w)}$. Since $L_0\otimes_{\QQ_l}E\cong\prod_{\tau:L_0\hookrightarrow E}E$, we have a decomposition $D=\prod_{\tau:L_0\hookrightarrow E}D_{\tau}$ where $D_{\tau}=D\otimes_{L_0\otimes_{\QQ_l}E,\tau\otimes 1}E$, and an induced Weil--Deligne representation $(r_{\tau}, N_{\tau})$ of dimension $[L_0:\QQ_l]$ over $E$. By \cite[2.2.1]{BM02}, the isomorphism class of $(r_{\tau}, N_{\tau})$ is independent of $\tau$, so we may denote any element in the isomorphism class by $\WD(D)$. 

Now let $\rho:G_K\to\GL_d(E)$ be a continuous potentially semistable representation. Choose $L/K$ such that $\rho|_{G_L}$ is semistable. There is a standard way to associate it with a $(\varphi,N,G_{L/K})$-module $$D=D_{\text{st},L}(\rho)=(B_{\text{st}}\otimes_{\QQ_l}V_{\rho})^{G_L}.$$ Still by \cite[\S 2.2.1]{BM02}, the isomorphism class of $\WD(D_{\text{st},L}(\rho))$ is independent of the choice of $L$. Hence we may set $\WD(\rho)=\WD(D_{\text{st},L}(\rho))$.

\begin{rmk}
Unlike in the $l\ne p$ case, in the $l=p$ case the $\WD$ functor is not faithful: it forgets the information from the filtration.
\end{rmk}

Next we explain some important results used in our proof. The first is an automorphy lifting theorem \cite[Theorem~2.1]{AN19}, which is essentially a special case of \cite[Theorem~6.1.1]{ACC+18}.

\begin{theorem}\label{automorphy_lifting}
Let $F$ be a CM field and let $l\ge5$ be a prime that is unramified in $F$. Let $\rho: G_F\to \GL_2(\overline{\QQ}_l)$ be a continuous representation satisfying the following conditions:
\begin{enumerate}
    \item $\rho$ is unramified almost everywhere.
    \item For each place $v\mid l$ of $F$, the representation $\rho|_{G_{F_v}}$ is crystalline with labelled Hodge--Tate weights all equal to $\{0, 1\}$.
    \item\label{tech} $\overline{\rho}$ is decomposed generic, $\overline{\rho}|_{G_{F(\zeta_l)}}$ is absolutely irreducible with enormous image. There exists $\sigma\in G_F-G_{F(\zeta_l)}$ such that $\overline{\rho}(\sigma)$ is a scalar.
    \item There exists a regular algebraic cuspidal automorphic representation $\pi$ of $\GL_2(\AAA_F)$ of weight $0$ satisfying \begin{enumerate}
        \item $\overline{\rho}\cong\overline{r}_{\iota}(\pi)$.
        \item For each place $v\mid l$ of $F$, $\pi_v$ is unramified.
    \end{enumerate}  
\end{enumerate}

Then $\rho$ is automorphic: there exists a cuspidal automorphic representation $\Pi$ of $\GL_2(\AAA_F)$ of weight $0$ such that $\rho\cong r_{\iota}(\Pi)$. Moreover, if $v$ is a finite place of $F$ and either $v\mid l$ or both $\rho$ and $\pi$ are unramified at $v$, then $\Pi_v$ is unramified.
\end{theorem}

We give the definitions of decomposed generic and enormous image in the technical condition \ref{tech}. Take a continuous mod $l$ representation $\overline{\rho}:G_F\to\GL_2(\overline{\FF}_l)$.

\begin{defn}
Let $\ad^0$ denote the space of trace zero matrices in $\M_{2\times 2}(\overline{\FF}_l)$ with the adjoint $\GL_2(\overline{\FF}_l)$-action. We say that $\overline{\rho}$ has \emph{enormous image} if $H=\overline{\rho}(G_F)$ is absolutely irreducible and satisfies the following:
\begin{enumerate}
    \item $H$ has no nontrivial $l$-power order quotient.
    \item $H^0(H, \ad^0) = H^1(H, \ad^0) = 0$.
    \item For any simple $\overline{\FF}_l[H]$-submodule $W\subseteq\ad^0$, there is a regular semisimple $h\in H$ such that $W^h\ne0$.
\end{enumerate}
\end{defn}

\begin{defn}
\begin{enumerate}
    \item We say that a prime $p\ne l$ is \emph{decomposed generic} for $\overline{\rho}$ if $p$ splits completely in $F$, and for any place $v\mid p$ in $F$, $\overline{\rho}$ is unramified at $v$ with the quotient of the two eigenvalues of $\overline{\rho}(\Frob_v)$ not equal to $1,p,p^{-1}$. Here $\Frob_v$ is the geometric Frobenius element at $v$.
    \item We say that $\overline{\rho}$ is \emph{decomposed generic} if there exists $p\ne l$ that is decomposed generic for $\overline{\rho}$.
\end{enumerate}
\end{defn}

Next is a potential automorphy theorem \cite[Theorem~3.9]{AN19}.

\begin{theorem}\label{potential_automorphy}
Suppose that $F$ is a CM field, $l$ is an odd prime that is unramified in $F$, and $k/\FF_l$ finite. Let $\calO$ be a discrete valuation ring finite over $W(k)$ with residue field $k$. Let $\overline{\rho}:G_F\to\GL_2(k)$ be a continuous absolutely irreducible representation such that 
\begin{itemize}
    \item $\det(\overline{\rho})=\overline{\epsilon}_l^{-1}$;
    \item for each $v\mid l$, $\overline{\rho}|_{G_{F_v}}$ admits a crystalline lift $\rho_v:G_{F_v}\to\GL_2(\calO)$ with all labelled Hodge--Tate weights equal to $\{0, 1\}$.
\end{itemize}

Suppose moreover that $F^\text{avoid}/F$ is a finite extension. Then we can find
\begin{itemize}
    \item a finite CM extension $F_1/F$ that is linearly disjoint from $F^\text{avoid}$ over $F$ with $l$ unramified in $F_1$;
    \item a regular algebraic cuspidal automorphic representation $\pi$ for $GL_2(\AAA_{F_1})$ of weight $0$, unramified at places above $l$;
    \item an isomorphism $\iota:\overline{\QQ}_l\xrightarrow{\sim}\CC$
\end{itemize}
such that (composing $\overline{\rho}$ with some embedding $k\hookrightarrow\overline{\FF}_l$) $$\overline{r}_\iota(\pi)\cong\overline{\rho}|_{G_{F_1}}.$$

If $\overline{v}_0\nmid l$ is a finite place of $F^+$, then we can moreover find $F_1$ and $\pi$ as above with $\pi$ unramified above $\overline{v}_0$.
\end{theorem}

Finally we state two remarkable recent results. The first is \cite[Theorem~4.3.3]{ACa23}. The original theorem is proved for $\pi$ of rank $n$ and arbitrary weight, but for simplicity we only state the special case when $\pi$ is of rank $2$ and weight $0$.

\begin{theorem}\label{deRham}
Let $F$ be a CM field,  $\iota:\overline{\QQ}_l\to\CC$ an isomorphism and $\pi$ a regular algebraic cuspidal automorphic representation of $\GL_2(\AAA_F)$ of weight $0$. If $\overline{r}_\iota(\pi)$ is absolutely irreducible and decomposed generic, then for every place $v\mid l$ of $F$ the representation $r_\iota(\pi)|_{G_{F_v}}$ is potentially semistable with $\tau$-labelled Hodge--Tate weights $\HT_\tau=\{0, 1\}$ for each embedding $\tau:F\hookrightarrow\overline{\QQ}_l$. If $\pi_v$ and $\pi_{v^c}$ are unramified, then $r_\iota(\pi)|_{G_{F_v}}$ is crystalline.
\end{theorem}

The second is the up-to-monodromy $l=p$ local-global compatibility \cite[Theorem~1.1, Remark~1.2]{Hev23}.

\begin{theorem}\label{semisimplification}
Let $F$ be a CM field,  $\iota:\overline{\QQ}_l\to\CC$ an isomorphism and $\pi$ a regular algebraic cuspidal automorphic representation of $\GL_n(\AAA_F)$. If $\overline{r}_\iota(\pi)$ is absolutely irreducible and decomposed generic, then for every place $v\mid l$ of $F$, $$\WD(r_{\iota}(\pi)|_{G_{F_v}})^{\Ss}\otimes_\iota\CC\cong\rec_{F_v}(\pi_v\otimes|\det|^{\frac{1-n}{2}})^{\Ss}.$$
\end{theorem}

We also modify the statement for simplicity. The original theorem proves that under certain partial order, the monodromy of the Galois side is ``less than'' the monodromy of the automorphic side (\cite[Remark~1.2]{Hev23}). In particular, if the Galois side has nontrivial monodromy, so does the automorphic side. It remains to prove the converse.

\section{Local-Global Compatibility}
Now we are ready to prove our main theorem.

\begin{theorem}\label{main}
Let $F$ be a CM field and let $\pi$ be a regular algebraic cuspidal automorphic representation of $\GL_2(\AAA_F)$ of weight $0$. Fix an isomorphism $\iota:\overline{\QQ}_l\cong\CC$. Suppose that 
\begin{enumerate}
    \item $l\ge5$ is a prime that is unramified in $F$,
    \item $\overline{r}_{\iota}(\pi)$ is decomposed generic, $\overline{r}_{\iota}(\pi)(G_{F(\zeta_l)})$ is enormous, and there is a $\sigma\in G_F-G_{F(\zeta_l)}$ such that $\overline{r}_{\iota}(\pi)(\sigma)$ is a scalar.
\end{enumerate}
Then for any finite place $v\mid l$ in $F$, we have
$$\WD(r_{\iota}(\pi)|_{G_{F_v}})^{\Fss}\otimes_\iota\CC\cong\rec_{F_v}(\pi_v\otimes|\det|^{-1/2}).$$
\end{theorem}
\begin{proof}
Fix a prime $p\ne l$ for which $\overline{r}_{\iota}(\pi)$ is decomposed generic. By Theorem \ref{deRham} we know that $r_{\iota}(\pi)|_{G_{F_v}}$ is potentially semistable, so it makes sense to consider the associated Weil--Deligne representation. By Theorem \ref{semisimplification}, it suffices to show that if $v\mid l$ is a finite place at which $\pi$ is special (i.e., $\pi_v$ is a twist of Steinberg), then $\WD(r_{\iota}(\pi)|_{G_{F_v}})$ has nontrivial monodromy. Let $N$ be the monodromy operator. To show that $N$ is nontrivial it suffices to do so after restricting to a finite extension. In particular, we may go to a solvable base change that is disjoint from $\overline{F}^{\ker(\overline{r}_{\iota}(\pi))}$ in which $l$ is unramified and $p$ is totally split, and assume that 
\begin{itemize}
    \item $\pi_v$ is an unramified twist of Steinberg,
    \item $r_{\iota}(\pi)|_{G_{F_v}}$ is semistable,
    \item $\overline{r}_{\iota}(\pi)$ is unramified at $v$ and $v^c$.
\end{itemize}

Now assume that $N=0$, and therefore $r_{\iota}(\pi)|_{G_{F_v}}$ is crystalline. We apply Theorem \ref{potential_automorphy} with $\rho=r_{\iota}(\pi)$ and $F^{\text{avoid}}$ equal to the Galois closure of $\overline{F}^{\ker(\overline{r}_{\iota}(\pi))}(\zeta_l)/\QQ$, to get a CM Galois extension $F_1/F$ linearly disjoint from $F^\text{avoid}$ over $F$ with $l$ unramified in $F_1$, and a regular algebraic cuspidal automorphic representation $\pi_1$ for $\GL_2(\AAA_{F_1})$ unramified at places above $l$ and of weight $0$, such that $\overline{r}_{\iota}(\pi)|_{G_{F_1}}\cong\overline{r}_{\iota}(\pi_1)$.  

We wish to apply the automorphy lifting theorem Theorem \ref{automorphy_lifting}. To do so, we check the following assumptions:
\begin{enumerate}
\setcounter{enumi}{-1}
    \item The prime $l$ is unramified in $F_1$. This is guaranteed by Theorem \ref{potential_automorphy}.
    \item $r_{\iota}(\pi)|_{G_{F_1}}$ is unramified almost everywhere. We know that $r_{\iota}(\pi)$ is unramified almost everywhere by the main result of \cite{HLTT16} and so is the restriction.
    \item For each place $v_1\mid l$ of $F_1$, the representation $r_{\iota}(\pi)|_{G_{F_{1,v_1}}}$ is crystalline with labelled Hodge--Tate weights all equal to $\{0, 1\}$. Since $l$ is unramified in both $F$ and $F_1$ and hence $(F_v)_0=F_v$, $(F_{1,v_1})_0=F_{1,v_1}$, writing $V$ for the $\QQ_l$-vector space underlying $r_{\iota}(\pi)|_{G_{F_v}}$, we can use Galois descent to compute the dimensions 
    \begin{align*}
        \dim_{\QQ_l}V & =\dim_{F_v}(B_{\text{cris}}\otimes_{\QQ_l}V)^{G_{F_v}}\\
         & =\dim_{F_v}((B_{\text{cris}}\otimes_{\QQ_l}V)^{G_{F_{1,v_1}}})^{\Gal(F_{1,v_1}/F_v)}\\
         & =\dim_{F_{1,v_1}}(B_{\text{cris}}\otimes_{\QQ_l}V)^{G_{F_{1,v_1}}}.
    \end{align*} 
    Thus the restriction $r_{\iota}(\pi)|_{G_{F_{1,v_1}}}$ must also be crystalline.
    \item $\overline{r}_{\iota}(\pi)|_{G_{F_1}}$ is decomposed generic (by a similar argument using Chebotarev density to choose a prime at which $\overline{r}_{\iota}(\pi)|_{G_{F_1}}$ is decomposed generic, as in the proof of \cite[Theorem~4.1]{AN19}). $\overline{r}_{\iota}(\pi)|_{G_{F_1(\zeta_l)}}$ is absolutely irreducible (encoded in the definition of enormous image) with enormous image (by our choice of $F^\text{avoid}$). There exists $\sigma\in G_{F_1}-G_{F_1(\zeta_l)}$ such that $\overline{r}_{\iota}(\pi)(\sigma)$ is a scalar (still by our choice of $F^\text{avoid}$).
    \item There exists a regular algebraic cuspidal automorphic representation $\pi_1$ of $\GL_2(\AAA_{F_1})$ of weight $0$ such that $\overline{r}_{\iota}(\pi_1)\cong \overline{r}_{\iota}(\pi)|_{G_{F_1}}$. For each place $v_1\mid l$ of $F_1$, $\pi_{1,v_1}$ is unramified. This is by Theorem \ref{potential_automorphy}.
\end{enumerate}

The automorphy lifting theorem gives a cuspidal automorphic representation $\Pi$ of $\GL_2(\AAA_{F_1})$ of weight $0$ such that $r_{\iota}(\Pi)\cong r_{\iota}(\pi)|_{G_{F_1}}$ and $\Pi_{v_1}$ is unramified at all $v_1\mid l$ in $F_1$. Now 
$$\rec_{F_{1,v_1}}(\Pi_{v_1}\otimes|\det|^{-1/2})\cong\WD(r_{\iota}(\Pi)|_{G_{F_{1,v_1}}})^{\Ss}\otimes_{\iota}\CC \cong\WD(r_{\iota}(\pi)|_{G_{F_{1,v_1}}})^{\Ss}\otimes_{\iota}\CC.$$ Under the local Langlands correspondence for $F_{1,v_1}$, the last object corresponds to a subquotient of an induction of two characters whose quotient equal to $|\cdot|^{\pm1}$, which can not be generic. This contradicts the fact that $\Pi_{v_1}$ is generic.
\end{proof}

\begin{proof}[Proof of Theorem \ref{density}]
It follows directly from Theorem \ref{main} and \cite[Lemma~2.9]{AN19}.
\end{proof}

\section{An application}
In \cite[\S 9]{Maz94}, Mazur defined a collection of invariants attached to the isomorphism classes of certain filtered $(\varphi,N)$-modules called the two-dimensional monodromy modules. We first recall the definitions. Let $D$ be a filtered $(\varphi, N)$-module, which is a $K_0$-vector space with actions of $\varphi$ and $N$ and an exhaustive and separated decreasing filtration $\{\Fil^i\}_i$ on $D_K:=D\otimes_{K_0}K$, satisfying the following:
\begin{enumerate}
    \item $\varphi$ is $K_0$-semilinear;
    \item $N$ is $K_0$-linear;
    \item $N\varphi=l\varphi N$.
\end{enumerate}
Note that $D$ has a basis of $\varphi$-eigenvectors, and $N$ takes $\varphi$-eigenvectors to $\varphi$-eigenvectors or $0$.

\begin{defn}
A \emph{(two-dimensional) monodromy module} $D$ is a two-dimensional filtered $(\varphi, N)$-module with the following properties:
\begin{enumerate}
    \item $N$ is nonzero;
    \item There is $j_0\in\ZZ$ for which $\Fil^{j_0}D_K$ is $1$-dimensional;
    \item $N(D)\otimes_{K_0}K\ne\Fil^{j_0}D_K$.
\end{enumerate}
\end{defn}

\begin{defn}
Let $D$ be a two-dimensional monodromy module. The \emph{Fontaine--Mazur $\calL$-invariant} $\calL(D)$ is the unique element of $K$ such that $x-\calL(D)\cdot Nx$ spans $\Fil^{j_0}D_K$, where $x$ is a $\varphi$-eigenvector with $Nx\ne 0$.
\end{defn}

Let $f$ be a classical normalized eigenform of $\Gamma_0(N)$ of even weight $k\ge2$ (and Nebentypus character $\chi$, for some integer $N$), and $\rho_f:G_{\QQ}\to\GL_2(E)$ its associated $l$-adic Galois representation for a suitable $E/\QQ_l$. If $\rho_{f,l}:=\rho_f|_{G_{\QQ_l}}$ is semistable non-crystalline, it is automatically non-critical, which means that $D_{\text{st}}(\rho_{f,l})$ is a monodromy module, and therefore we may make sense of the Fontaine--Mazur $\calL$-invariant $\calL(D_{\text{st}}(\rho_{f,l}))$ attached to $f$. Note that this $\calL$-invariant only depends on $\rho_{f,l}$, so in particular it only depends on $f$ and $l$.

Now we return to the settings of this paper and slightly generalize the above idea. Let $F$ be a CM field and let $l$ be a rational prime that is inert in $F$. Fix an isomorphism $\iota:\overline{\QQ}_l\cong\CC$. Let $\pi$ be a regular algebraic cuspidal automorphic representation of $\GL_2(\AAA_F)$ of weight $0$. Assume that $\pi_v$ is special for the finite place $v$ in $F$ lying above $l$. The proof of Theorem \ref{main} shows that $r_{\iota}(\pi)|_{G_{F_v}}$ is (potentially) semistable non-crystalline, and is non-critical since $G_{F_v}\cong G_{\QQ_l}$. Hence we can make sense of the Fontaine--Mazur $\calL$-invariant $\calL(D_{\text{st}}(r_{\iota}(\pi)|_{G_{F_v}}))$ attached to $\pi$. Similar to the classical modular form case, the $\calL$-invariant only depends on $r_{\iota}(\pi)|_{G_{F_v}}$, so in particular it only depends on $\pi$ and $l$. More generally, if we only assume that $l$ is unramified in $F$, then we need to assume that $r_{\iota}(\pi)|_{G_{F_v}}$ is non-critical to make $D_{\text{st}}(r_{\iota}(\pi)|_{G_{F_v}})$ a monodromy module. 

Once we have the definition, it is a natural question to ask if it coincides with other versions of $\calL$-invariants other than the one due to Fontaine--Mazur. In the classical modular form case the question is extensively studied and the answer is yes (for instance see the introduction of \cite{BDI10}), but in general the question is widely open. We hope to see the question being answered with the emergence of generalizations of more versions of $\calL$-invariants. 

\bibliographystyle{amsalpha}
\bibliography{bibtex}

\newcommand{\etalchar}[1]{$^{#1}$}
\providecommand{\bysame}{\leavevmode\hbox to3em{\hrulefill}\thinspace}
\providecommand{\MR}{\relax\ifhmode\unskip\space\fi MR }
\providecommand{\MRhref}[2]{%
  \href{http://www.ams.org/mathscinet-getitem?mr=#1}{#2}
}
\providecommand{\href}[2]{#2}
\begin{thebibliography}{BLGGT14b}

\bibitem[A'C24]{ACa23}
Lambert A'Campo, \emph{Rigidity of automorphic {G}alois representations over {CM} fields}, Int. Math. Res. Not. IMRN (2024), no.~6, 4541--4623. \MR{4721650}

\bibitem[ACC{\etalchar{+}}23]{ACC+18}
Patrick~B. Allen, Frank Calegari, Ana Caraiani, Toby Gee, David Helm, Bao~V. Le~Hung, James Newton, Peter Scholze, Richard Taylor, and Jack~A. Thorne, \emph{Potential automorphy over {CM} fields}, Ann. of Math. (2) \textbf{197} (2023), no.~3, 897--1113. \MR{4564261}

\bibitem[All16]{All16}
Patrick~B. Allen, \emph{Deformations of polarized automorphic {G}alois representations and adjoint {S}elmer groups}, Duke Math. J. \textbf{165} (2016), no.~13, 2407--2460. \MR{3546966}

\bibitem[AN20]{AN19}
Patrick~B. Allen and James Newton, \emph{Monodromy for some rank two {G}alois representations over {CM} fields}, Doc. Math. \textbf{25} (2020), 2487--2506. \MR{4213130}

\bibitem[BDI10]{BDI10}
Massimo Bertolini, Henri Darmon, and Adrian Iovita, \emph{Families of automorphic forms on definite quaternion algebras and {T}eitelbaum's conjecture}, Ast\'{e}risque (2010), no.~331, 29--64. \MR{2667886}

\bibitem[BLGGT12]{BLGGT12}
Thomas Barnet-Lamb, Toby Gee, David Geraghty, and Richard Taylor, \emph{Local-global compatibility for {$l=p$}, {I}}, Ann. Fac. Sci. Toulouse Math. (6) \textbf{21} (2012), no.~1, 57--92. \MR{2954105}

\bibitem[BLGGT14a]{BLGGT14a}
\bysame, \emph{Local-global compatibility for {$l=p$}, {II}}, Ann. Sci. \'{E}c. Norm. Sup\'{e}r. (4) \textbf{47} (2014), no.~1, 165--179. \MR{3205603}

\bibitem[BLGGT14b]{BLGGT14b}
\bysame, \emph{Potential automorphy and change of weight}, Ann. of Math. (2) \textbf{179} (2014), no.~2, 501--609. \MR{3152941}

\bibitem[BLGHT11]{BLGHT11}
Tom Barnet-Lamb, David Geraghty, Michael Harris, and Richard Taylor, \emph{A family of {C}alabi-{Y}au varieties and potential automorphy {II}}, Publ. Res. Inst. Math. Sci. \textbf{47} (2011), no.~1, 29--98. \MR{2827723}

\bibitem[BM02]{BM02}
Christophe Breuil and Ariane M\'{e}zard, \emph{Multiplicit\'{e}s modulaires et repr\'{e}sentations de {${\rm GL}_2({\bf Z}_p)$} et de {${\rm Gal}(\overline{\bf Q}_p/{\bf Q}_p)$} en {$l=p$}}, Duke Math. J. \textbf{115} (2002), no.~2, 205--310, With an appendix by Guy Henniart. \MR{1944572}

\bibitem[Car14]{Car14}
Ana Caraiani, \emph{Monodromy and local-global compatibility for {$l=p$}}, Algebra Number Theory \textbf{8} (2014), no.~7, 1597--1646. \MR{3272276}

\bibitem[Fon94]{Fon94}
Jean-Marc Fontaine, \emph{Repr\'{e}sentations {$l$}-adiques potentiellement semi-stables}, Ast\'{e}risque (1994), no.~223, 321--347, P\'{e}riodes $p$-adiques (Bures-sur-Yvette, 1988). \MR{1293977}

\bibitem[{Hev}23]{Hev23}
Bence {Hevesi}, \emph{{Ordinary parts and local-global compatibility at $\ell=p$}}, arXiv e-prints (2023), arXiv:2311.13514.

\bibitem[HLTT16]{HLTT16}
Michael Harris, Kai-Wen Lan, Richard Taylor, and Jack Thorne, \emph{On the rigid cohomology of certain {S}himura varieties}, Res. Math. Sci. \textbf{3} (2016), Paper No. 37, 308. \MR{3565594}

\bibitem[HSBT10]{HSBT10}
Michael Harris, Nick Shepherd-Barron, and Richard Taylor, \emph{A family of {C}alabi-{Y}au varieties and potential automorphy}, Ann. of Math. (2) \textbf{171} (2010), no.~2, 779--813. \MR{2630056}

\bibitem[HT01]{HT01}
Michael Harris and Richard Taylor, \emph{The geometry and cohomology of some simple {S}himura varieties}, Annals of Mathematics Studies, vol. 151, Princeton University Press, Princeton, NJ, 2001, With an appendix by Vladimir G. Berkovich. \MR{1876802}

\bibitem[Luu15]{Luu15}
Martin Luu, \emph{Deformation theory and local-global compatibility of {L}anglands correspondences}, Mem. Amer. Math. Soc. \textbf{238} (2015), no.~1123, vii+101. \MR{3402383}

\bibitem[{Mat}23]{Mat23}
Kojiro {Matsumoto}, \emph{{On the potential automorphy and the local-global compatibility for the monodromy operators at $p \neq l$ over CM fields}}, arXiv e-prints (2023), arXiv:2312.01551.

\bibitem[Maz94]{Maz94}
B.~Mazur, \emph{On monodromy invariants occurring in global arithmetic, and {F}ontaine's theory}, {$p$}-adic monodromy and the {B}irch and {S}winnerton-{D}yer conjecture ({B}oston, {MA}, 1991), Contemp. Math., vol. 165, Amer. Math. Soc., Providence, RI, 1994, pp.~1--20. \MR{1279599}

\bibitem[Qia23]{Qia23}
Lie Qian, \emph{Potential automorphy for {$GL_n$}}, Invent. Math. \textbf{231} (2023), no.~3, 1239--1275. \MR{4549090}

\bibitem[Sch15]{Sch15}
Peter Scholze, \emph{On torsion in the cohomology of locally symmetric varieties}, Ann. of Math. (2) \textbf{182} (2015), no.~3, 945--1066. \MR{3418533}

\bibitem[Tay02]{Tay02}
Richard Taylor, \emph{Remarks on a conjecture of {F}ontaine and {M}azur}, J. Inst. Math. Jussieu \textbf{1} (2002), no.~1, 125--143. \MR{1954941}

\bibitem[Var24]{Var14}
Ila Varma, \emph{Local-global compatibility for regular algebraic cuspidal automorphic representations when {$\ell \ne p$}}, Forum Math. Sigma \textbf{12} (2024), Paper No. e21, 32. \MR{4706792}

\bibitem[{Yan}21]{Yan21}
Yuji {Yang}, \emph{{An Ordinary Rank-Two Case of Local-Global Compatibility for Automorphic Representations of Arbitrary Weight Over CM Fields}}, arXiv e-prints (2021), arXiv:2111.00318.

\end{thebibliography}

\end{document}